\documentclass{amsart}
\usepackage{amssymb,amsmath,amsfonts,amsthm,fancyhdr}

\newcommand{\spp}{\kappa}

\newcommand{\Ann}{\mathop{\mathrm{Ann}}\nolimits}
\newcommand{\Supp}{\mathop{\mathrm{Supp}}\nolimits}

\newcommand{\Spec}{\mathop{\mathrm{Spec}}\nolimits}
\newcommand{\depth}{\mathop{\mathrm{depth}}\nolimits}

\newcommand{\8}{{\infty}}
\newcommand{\pd}{\operatorname{pd}}
\newcommand{\Min}{\operatorname{Min}}

\newcommand{\rk}{\operatorname{rank}}

\newcommand{\Tor}{\mathop{\mathrm{Tor}}\nolimits}
\newcommand{\Ext}{\mathop{\mathrm{Ext}}\nolimits}

\newcommand{\length}{\mathrm{\ell}}

\renewcommand{\i}[1]{\mathfrak{#1}}

\newtheorem{thm}{Theorem}[section]
\newtheorem{cor}[thm]{Corollary}

\newtheorem{lemma}[thm]{Lemma}

\theoremstyle{definition}
\newtheorem{defn}[thm]{Definition}

\theoremstyle{remark}
\newtheorem{rmk}[thm]{Remark}

 \numberwithin{equation}{thm}

\begin{document}
\title{Frobenius criteria of freeness and Gorensteinness}
\author{Jinjia Li}
\address{Department of Mathematics, University of Louisville, Louisville, Kentucky 40292}
\email{jinjia.li@louisville.edu}
%\thanks{}
\subjclass{Primary: 13A35; Secondary: 13D07, 13H10.}

\date{\today}

\keywords{finite projective dimension, Frobenius, Ext, Tor, canonical module, maximal Cohen-Macaulay module, Gorenstein ring, rigidity, intersection multiplicity.}

\begin{abstract}Let $F^n(-)$ be the Frobenius functor of Peskine and Szpiro. In this note, we show that the maximal Cohen-Macaulayness of $F^n(M)$ forces $M$ to be free, provided $M$ has a rank. We apply this result to obtain several Frobenius related criteria for the Gorensteinness of a local ring $R$, one of which improves a previous characterization due to Hanes and Huneke. We also establish a special class of finite length modules over Cohen-Macaulay rings, which are rigid against Frobenius.  

\end{abstract}

\maketitle

%%%%%%%%%%%%%%%%%%%%%%%%%%%%%%%%%%
\section{Introduction and notation}
%%%%%%%%%%%%%%%%%%%%%%%%%%%%%%%%%%

We assume throughout that $R$ is a commutative Noetherian local
ring in prime characteristic $p>0$. All modules are assumed to be
finitely generated unless otherwise specified. The Frobenius
endomorphism $f\colon R \to R$ is defined by $f(r)=r^p$ for $r \in
R$. Each iteration $f^n$ defines a new $R$-module structure on
$R$, denoted by ${}^{f^n}\!\! R$, via restriction of scalars.

Next we recall the Frobenius functor introduced by Peskine and Szpiro, see \cite{PS}.
For any $R$-module $M$, $F_R^n(M)$ denotes the base change $M\otimes_R {}^{f^n}\!\! R$ along $f^n$; note that the module structure is via usual multiplication in $R$ on the right hand factor of the tensor product. However, it is easy to see that the depth and dimension are unchanged if one were to view it as an $R$-module via the left factor instead. Its derived functors $\Tor^R_i (M,{}^{f^n}\!\! R)$ are similarly viewed as $R$-modules via the target of the base change map $f^n$. We omit the subscript $R$ if there is no ambiguity about $R$.
For convenience, we frequently use $q$ to denote the power $p^n$,
which may vary. Note that $F^n(R) \cong R$. When $M$ is a cyclic module $R/I$, it is easy to
show that $F^n(R/I) \cong R/I^{[q]}$, where $I^{[q]}$ denotes the
ideal generated by the $q$-th powers of the generators of $I$. We use the notation $\bold x$ for a sequence of elements of $R$ and often write simply $R/\bold x$ for $R/(\bold x)$ to save space. Likewise, $\bold x^q$ denotes the ideal generated by the $q$th powers of the sequence $\bold x$. By codimension of $M$,  we mean $\dim R-\dim M$. 

In this note, we study the question about when the vanishing of the module $\Tor^R_i (M,{}^{f^n}\!\! R)$ detects the finiteness of the projective dimension of $M$ (we call such a module \textit{rigid against Frobenius}, see Definition~\ref{rigid}).  In \cite{AM}, Avramov and Miller showed that over a complete intersection ring, every module is rigid against Frobenius. The first example of a module which is not rigid against Frobenius  over a Gorenstein ring was obtained by Dao, Miller and the author in \cite{DLM}. However, the question of how to detect such a module in general is still quite open.  In particular, it is not known whether a finite length module over a Gorenstein ring is rigid against Frobenius. If this were true, then the main result in \cite{KV} could be extended from the complete intersection case to the Gorenstein case. In this note, we show a special case of this is true. Specifically, if a module takes the form $M/\bold x M$ where $M$ is a Cohen-Macaulay module of codimension 0 or 1,  and $\bold x$ is a maximal regular sequence of both $M$ and $R$, then it is rigid against Frobenius. Related to the establishment of the above result, we also provides several characterizations of Gorensteinness of the ring, concerning the vanishing of Ext or Tor involving Frobenius. Some of them improve similar characterizations obtained previously by Hanes and Huneke \cite{HH} and by Goto \cite{G}. 

For the proofs of the results in the next section, the properties of the higher
Euler characteristics of Koszul complexes are used in an essential
way. We recall some terms and results here. In the sequel, we use $\length(M)$, $\pd M$, and $\mu(M)$, respectively, to denote the
\emph{length}, \emph{projective dimension}, and \emph{minimal number of generators}, respectively,  of the module
$M$. 

Recall that the $i$th higher Euler characteristics for a pair of modules $M$,
$N$ such that $\length(M\otimes N)<\8$ and $\pd N <\8$ is defined
by the following formula
\[\chi_i(M,N)=
\sum_{j=i}^{\pd N}(-1)^{j-i}\length(\Tor_j^R(M,N)).
\]
By convention, $\chi(M,N)=\chi_0(M,N)$. Some standard facts about
$\chi$ and $\chi_i$ can be found in \cite{L, S}. In this paper, we
particularly need the following well-known results  due to Serre and
Lichtenbaum  \cite[Lemma 1 and Theorem 1]{L}, \cite[Chap. IV.: A. and Appendix II]{S}.

\begin{lemma}\label{Li1} Let $M$ be an $R$-module and $\bold x=\{x_1,x_2,...,x_c\}$
an $R$-sequence such that $\length(M/\bold xM)<\8$. Then $\chi(M,
R/\bold x)\geq 0$, with the equality holding iff $\dim M < c$.
\end{lemma}

\begin{thm} \label{Li2} Let $M$ be an $R$-module and $\bold x$
an $R$-sequence such that $\length(M/\bold xM)<\8$. Then for any
$i>0$, $\chi_i(M, R/\bold x) \geq 0$, with the equality holding
iff $\Tor_i(M,R/\bold x)=0$ {\rm{(}}hence $\Tor_j(M,R/\bold x)=0$ for all
$j\geq i$ by the $\Tor$ rigidity of $R/\bold x${\rm{)}}.
\end{thm}

%%%%%%%%%%%%%%%%%%%%%%%%%%%%%%%%%%
\section{Main results}
%%%%%%%%%%%%%%%%%%%%%%%%%%%%%%%%%%
\begin{defn}\label{rigid} We say a module $M$ is \textit{rigid against Frobenius} if either $M$ has finite projective dimension or $\Tor_i(M,  {}^{f^n}\!\! R) \neq 0$ for all $i$, and all $n>>0$.
\end{defn}

Avramov and Miller \cite{AM} (see also \cite{D}) showed that over a complete intersection ring, every module is rigid against Frobenius. Such a result improves previous results of Herzog \cite{He}, and of Koh and Lee \cite{KL} about detecting the finiteness of $\pd M$ by the vanishing of $\Tor_i(M,  {}^{f^n}\!\! R)$ for more than one $i>0$. Notice that the results of Herzog and of Koh and Lee do not require the complete intersection assumption. We refer to \cite{M2, DLM} for more background in detail. On the other hand,
Dao, Miller and the author \cite{DLM} demonstrated examples of modules over Gorenstein rings which are not rigid against Frobenius. Despite that, it is still interesting to ask what kinds of modules are rigid against Frobenius over a Cohen-Macaulay ring. For example, over a Gorenstein ring, is every finite length module rigid against Frobenius? If this were true, then the main result in \cite{KV} could be extended. A rather trivial example of a finite length module that is rigid against Frobenius is the residue field $k$ (or any finite direct sums of copies of $k$). The purpose of the note is to establish some other special classes of modules which are rigid against Frobenius.  We also apply the result to various situations to obtain: (1) some criteria for Gorensteinness and, (2) a result that can be regarded as a variation of the aforementioned result of Koh and Lee.

We begin by proving two easy lemmas about maximal Cohen-Macaulay
(henceforth MCM) modules with a rank.

\begin{lemma} \label{rank} Let $(R, \i m)$ be a local ring and $M$ an
$R$-module. Suppose $M$ has a rank and $\rk(M)=\mu(M)$, then $M$
is free.
\end{lemma}

\begin{proof}
Suppose that $M$ is not free. Let
\begin{equation}\label{pres}
R^t \overset{\alpha}{\to} R^r \to M \to 0
\end{equation}
be a minimal presentation of $M$, where $r=\rk(M)=\mu(M)$. One can
identify $\alpha$ with an $r \times t$ matrix $(\alpha_{ij})$
where $\alpha_{ij} \in \i m$. Let $S$ denote the set of all the
nonzerodivisors of $R$. Localizing (\ref{pres}) at $S$, since
$(S^{-1}R)^r \cong S^{-1}M$, one gets that $S^{-1}\alpha$ is the
zero map. Thus the images of $\alpha_{ij}$ in $S^{-1}R$ must all be
zero. But since $R$ embeds into $S^{-1}R$, $\alpha_{ij}$ must all
be zero, which is a contradiction.
\end{proof}

\begin{lemma} \label{length} Let $R$ be a Cohen-Macaulay local ring and $M$ an $R$-module with a rank. Let $\bold x$
be a full s.o.p. for $R$. Then $\chi(M, R/\bold x)=\rk(M)\length(R/\bold
x)$. If moreover $M$ is MCM, then $\length(M/\bold x
M)=\rk(M)\length(R/\bold x)$.
\end{lemma}

\begin{proof}
The case that $\rk(M)=0$ is trivial. Assume $\rk(M)>0$,
then $\Supp M=\Spec R$. Therefore, by
additivity of $\chi$, we have
\begin{align*}
 \chi(M, R/\bold x)  &= \sum_{\i p \in \Min R}\length(M_{\i p})\chi(R/\i p, R/\bold x)\\
                     &=\sum_{\i p \in \Min R}\rk(M)\length(R_{\i p})\chi(R/\i p, R/\bold x)\\
                     &=\rk(M)\chi(R,R/\bold x)\\
                     &=\rk(M)\length(R/\bold x),
\end{align*}
where $\Min R$ is the set of minimal primes.
If moreover $M$ is MCM, then $\Tor_i(M, R/\bold x)=0$ for $i>0$.
Thus, $\chi(M, R/\bold x)=\length(M/\bold x M)$.
\end{proof}

For the statements of the results in the rest of this paper, we define
the following invariant for a local Cohen-Macaulay ring $(R,\i m)$. Define
\[\spp(R)=\inf\{t|\text{ there is a s.o.p. $\bold x$ for $R$ such
that } {\i m}^{[p^t]} \subset (\bold x)\}\]

\begin{thm} \label{main1} Let $(R, \i m)$ be a Cohen-Macaulay local ring of positive dimension and of  characteristic $p>0$.
Let $M$ be a module over $R$ which has a rank. Assume $F^n(M)$ is
MCM for one $n \geq \spp(R)$. Then $M$ is free.
\end{thm}

\begin{proof}
Since $M$ has a rank, so does $F^n(M)$ for any $n$. Take a full s.o.p.
$\bold x$ for $R$. Then by Lemma~\ref{length},
\begin{align*}
\length(F^n_{R/\bold x}(M/{\bold x}M))&=\length(F_R^n(M)\otimes_R
R/\bold x)\\&=\rk(F_R^n(M)) \length(R/\bold x)\\&=\rk(M) \length(R/\bold
x).
\end{align*}
On the other hand, let \begin{equation}\label{pres1}
R^t \overset{\alpha}{\to} R^r \to M \to 0
\end{equation}
be a minimal presentation of $M$, where $r=\mu(M)$. One can tensor \ref{pres1} with ${}^{f^n}\!\! R$ to obtain a minimal presentation of $F^n(M)$:
\begin{equation}\label{pres2}
R^t \overset{\alpha^{[q]}}{\to} R^r \to F^n(M) \to 0.
\end{equation}

Since $n \geq \spp(R)$, all the entries in the matrix $\alpha^{[q]}$ are in $(\bold x )$. Thus we obtain the following exact sequence by
tensoring \ref{pres2} with $R/\bold x$
\begin{equation} 
(R/\bold x)^t \overset{0}{\to} (R/\bold x)^r \to F^n_{R/\bold x}(M/{\bold x}M) \to 0.
\end{equation}
Hence
\[
\length(F^n_{R/\bold x}(M/{\bold x}M))=\mu(M)\length(R/\bold x).
\]
It follows that $\rk(M)=\mu(M)$. Therefore by Lemma~\ref{rank},
$M$ is free.
\end{proof}

\begin{rmk}
The condition ``$M$ has a rank'' cannot be removed. See
\cite[2.1.7]{M2} for an example.
\end{rmk}

\noindent An immediate consequence of this is

\begin{cor}Let $R$ be a Cohen-Macaulay local ring of
characteristic $p>0$ with a canonical module $\omega$. Assume
$\omega$ has a rank (i.e., $R$ is generically Gorenstein). If
$F^n(\omega)$ is MCM for one $n \geq \spp(R)$, then $R$ is
Gorenstein.
\end{cor}

The following characterization of Gorenstein rings improves a
similar result of Hanes and Huneke \cite[2.9]{HH} by replacing a strong and
complicated condition there with the mild condition that ${}^{f^n}\!\! R$ has rank
(e.g., satisfied by domains). See \cite{G} and \cite{R} for another similar result originally due to Goto.

\begin{cor} \label{HH}
Let $R$ be a
Cohen-Macaulay, generically Gorenstein local ring of dimension $d$ and of characteristic $p>0$.
Suppose that there exists some $n \geq \spp(R)$ such that 
${}^{f^n}\!\! R$ is a finite $R$-module which has a rank and such that
$\Ext_R^i({}^{f^n}\!\! R, R)=0$ for all $1 \leq i \leq d$.
 Then $R$ is Gorenstein.
\end{cor}

\begin{proof}
We can assume $R$ is complete; hence it admits a canonical module
$\omega$. By Lemma 2.1 in \cite{HH}, we get that  $F^n(\omega)$ is MCM (note that the depth and dimension are independent of which of the two possible $R$-module structures on $M\otimes_R {}^{f^n}\!\! R$ is used).
It follows from the previous corollary that $R$ is Gorenstein.
\end{proof}

We next prove a generalization of Theorem~\ref{main1}. In some sense, it can be regarded as
a strengthening of the Koh and Lee's theorem for Cohen-Macaulay rings (see Theorem~2.2.8 of \cite{M2}). 
\begin{thm} \label{KL} Let $R$ be a Cohen-Macaulay local ring of dimension $d>0$ in
characteristic $p>0$. Let $M$ be a module over $R$ which has a
rank. Suppose for one $n \geq \spp(R)$, $\Tor_i(M,{}^{f^n}\!\! R)$
vanishes for all $1 \leq i \leq d-\depth F^n(M)$. Then $M$ has
finite projective dimension.
\end{thm}

\begin{proof}
We induct on $\depth F^n(M)$ decreasingly. The case $\depth
F^n(M)=d$ is nothing but Theorem~\ref{main1}. Assume the
proposition is established for all $R$-modules $N$ such that
$\depth F^n(N)=t+1 \leq d$. Let $M$ be a module such that $\depth
F^n(M)=t$. Take $T$ to be a first syzygy of $M$, i.e.,
there is a short exact sequence $0\to T \to R^r \to M \to 0$ where
$R^r$ is a free module. Tensoring with Frobenius gives an exact sequence
\[
0 \to F^n(T) \to R^r \to F^n(M) \to 0
\]
and that $\Tor_i(T,{}^{f^n}\!\! R)=0$ for $0<i \leq d-t-1$. Since $\depth F^n(T) \geq \depth F^n(M) +1$,
induction yields that $T$ has finite projective dimension. Thus $\pd M<\8$.
\end{proof}

We apply the above results to identify a class of finite length modules which are rigid against Frobenius.
\begin{cor}\label{free} Let $R$ be a Cohen-Macaulay local ring of dimension $d>0$ and of
characteristic $p>0$. Let $M$ be a maximal Cohen-Macaulay module over $R$ which has a rank.
Let $\bold x$ be any s.o.p. for $R$. Then the following are
equivalent:
\begin{enumerate}
\item $M$ is free,
\item $\length (F^n(M/\bold
xM))=q^d\length(M/\bold xM)$ for one $n \geq \spp(R)$ (hence all $n$),
\item $\Tor_i(M/\bold xM, {}^{f^n}\!\! R)$ vanishes for all $i,n >0$,
\item There exists one $i>0$ and one $n \geq \spp(R)$such that $\Tor_i(M/\bold xM, {}^{f^n}\!\! R)$ vanishes.
\end{enumerate}
\end{cor}

\begin{proof} The implications (1)$\Rightarrow $(2) and  (3)$\Rightarrow $(4) are obvious, and that (1) $\Rightarrow $(3) by \cite{PS}.\\
(4) $\Rightarrow $(1): 
If $i>1$, let $S_{i-1}$ be the $(i-1)$th syzygy of $M$. Consider the following truncation of the minimal free resolution of $M$,
\[0 \to S_{i-1} \to F_{i-2} \to \cdots \to F_1 \to F_0 \to M\to 0.\]
It is easy to check $S_{i-1}$ is MCM and the exactness of this sequence is preserved by tensoring with $R/\bold x$. Since for $i>0$,  the functor $\Tor_i(-, {}^{f^n}\!\! R)$ vanishes on modules of finite projective dimension \cite{PS}, it follows that 
\[\Tor_i(M/\bold x M, {}^{f^n}\!\! R)\cong \Tor_1(S_{i-1}/\bold x
S_{i-1}, {}^{f^n}\!\! R).\]
So we reduce to the case $i=1$.

Assume $i=1$. Let $G_\bullet$ be a minimal resolution for $M$
and $K_\bullet$ be the Koszul complex on $\bold x$. Then the total
complex of the double complex $G_\bullet \otimes K_\bullet$ gives a
minimal resolution of $M/\bold x M$ over $R$. Since
$H_j(H_i(F^n(G_\bullet)) \otimes F^n(K_\bullet))$ converges to
$H_{j+i}(F^n(G_{\bullet})\otimes F^n(K_\bullet)) \cong
H_{j+i}(F^n(G_{\bullet}\otimes K_\bullet))$, we obtain the following exact
sequence from the low degree terms of the spectral sequence,
\begin{equation*} 
\to \Tor_2(F^n(M), R/\bold x^q) \to \Tor_1(M, {}^{f^n}\!\! R) \otimes R/\bold x^q \to
\Tor_1(M/\bold x M, {}^{f^n}\!\! R) \to
\end{equation*}
\begin{equation} \label{ldt}
\to \Tor_1(F^n(M), R/\bold x^q) \to 0. 
\end{equation}
Therefore, $\Tor_1(M/\bold xM, {}^{f^n}\!\! R)=0$
implies $\Tor_1(F^n(M), R/\bold x^q)=0$ whence $F^n(M)$ is MCM. By
Theorem~\ref{main1}, $M$ is free.\\
(2) $\Rightarrow $(1): By Theorem~\ref{Li2}, we have
\begin{equation*} 
\length (F^n(M/\bold x M))=\chi(F^n(M),R/\bold x^q)+\chi_1(F^n(M),R/\bold x^q)
\end{equation*}
\begin{equation} \label{special}
\geq \chi(F^n(M),R/\bold x^q).
\end{equation} 
On the other hand, by Lemma~\ref{length} applied to $F^n(M)$ and $M$, one gets
\[
\chi(F^n(M),R/\bold x^q)=\rk(M)\length(R/\bold x^q)=q^d\rk(M)\length(R/\bold x)=q^d \length(M/\bold xM).
\]
Thus, (2) forces the inequality in (\ref{special})
to become an equality, and therefore one gets $\chi_1(F^n(M),R/\bold x^q)=0$.
Hence by Theorem~\ref{Li2}, $F^n(M)$ is maximal Cohen-Macaulay. Again Theorem
\ref{main1} forces $M$ to be free.
\end{proof}

\noindent Similarly, one can also prove

\begin{cor} Let $R$ be a Cohen-Macaulay local ring of dimension $d>0$ and of
characteristic $p>0$. Let $M$ a Cohen-Macaulay $R$-module of
codimension $1$. Let $\bold x$ be any s.o.p. for $M$ which is also
an $R$-sequence. Then the following are equivalent:
\begin{enumerate}
\item $\pd M < \infty$, 
\item $\Tor_i(M/\bold xM, {}^{f^n}\!\! R)$ vanishes for all $i,n >0$,
\item There exists one $i>0$ and one $n \geq \spp(R)$such that $\Tor_i(M/\bold xM, {}^{f^n}\!\! R)$ vanishes.
\end{enumerate}
\end{cor}

\begin{proof} (1) $\Rightarrow$ (2) is by \cite{PS} and (2) $\Rightarrow$ (3) is obvious.\\
(3) $\Rightarrow$ (1): Take some element $y \in \Ann_R M$ that is regular on $R$. Viewing $M$ as a module over $R/yR$ and taking the $(i-1)$th syzygy of $M$, we then reduce to the case of $i=1$.  Assume $i=1$, as in the proof of the Corollary~\ref{free}, we have the exact sequence (\ref{ldt}) of low degree terms. (3) implies that $\Tor_1(F^n(M), R/\bold x^q)=0$. By
the $\Tor$ rigidity of $R/\bold x^q$, $\Tor_i(F^n(M), R/\bold x^q)=0$ for all
$i>0$. Thus $\depth F^n(M)=d-1$. Also by the exact sequence
(\ref{ldt}) again and Nakayama's Lemma, $\Tor_1(M, {}^{f^n}\!\! R)=0$. It then
follows from Theorem~\ref{KL} that $\pd M<\8$ (Theorem~\ref{KL} is applicable here since the rank of $M$ is $0$).
\end{proof}

Applying Corollary~\ref{free} to the canonical module $\omega$,
one then gets another criterion for Gorensteinness:

\begin{cor} Let $R$ be a Cohen-Macaulay ring with canonical module
$\omega$. Assume $\omega$ has a rank. If for some  full s.o.p. $\bold x$ of $R$, there exists one $i>0$ and one $n\geq \spp(R)$ such that
\[\Tor_i(\omega/\bold x
\omega, {}^{f^n}\!\! R))=0,\] 
then $R$ is Gorenstein.
\end{cor}

\begin{rmk}The author does not know whether $\Tor_i(\omega, {}^{f^n}\!\!
R)=0$ (for one $i>0$) implies $R$ is Gorenstein.
\end{rmk}

\specialsection*{ACKNOWLEDGEMENTS}
I would like to thank the referee for pointing out numerous minor errors in the earlier versions of this paper.

\bibliographystyle{amsalpha}

\end{document}